\documentclass[12pt]{article}
\usepackage{mlmodern}
\usepackage{pdfrender,xcolor}
\usepackage{graphicx,color}
\usepackage{amsmath,amsfonts,amssymb,amsthm,MnSymbol}
\usepackage[hidelinks,colorlinks=true,linkcolor=blue,citecolor=blue]{hyperref}
\usepackage{rsfso}
\usepackage{comment}
\usepackage{enumerate}
\usepackage{algorithm}
\usepackage{algpseudocode}
\usepackage{mathrsfs}

\usepackage[numbers]{natbib}
\definecolor{darkgreen}{rgb}{0.06, 0.56, 0.2}

\newcommand{\R}{\mathbb{R}}

\DeclareMathOperator{\diag}{diag}

\newcommand{\sr}[1]{{\cal #1}}
\newcommand{\dd}[1]{\mathbb{#1}}

\newcommand{\br}[1]{\langle #1 \rangle}

\newcommand{\os}{\overset}

\newcommand{\eq}[1]{(\ref{eq:#1})}
\newcommand{\lem}[1]{Lemma~\ref{lem:#1}}

\newcommand{\thm}[1]{Theorem~\ref{thm:#1}}

\newcommand{\pro}[1]{Proposition~\ref{pro:#1}}

\newcommand{\exa}[1]{Example~\ref{exa:#1}}

\newcommand{\dfn}[1]{Definition~\ref{dfn:#1}}

\newcommand{\app}[1]{Appendix~\ref{app:#1}}
\newcommand{\sectn}[1]{Section~\ref{sec:#1}}

\newcommand{\thmt}[1]{\ref{thm:#1}}

\newcommand{\sect}[1]{\ref{sec:#1}}

\newcommand{\pend}{\hfill \thicklines \framebox(6.6,6.6)[l]{}}

\newenvironment{proof*}[1]{\noindent {\sc  #1} \rm}{\pend}

\newtheorem{theorem}{Theorem}[section]
\newtheorem{lemma}{Lemma}[section]

\newtheorem{proposition}{Proposition}[section]

\newtheorem{remark}{Remark}[section]

\newtheorem{corollary}{Corollary}[section]

\theoremstyle{definition}
\newtheorem{example}{Example}[section]
\newtheorem{definition}{Definition}[section]

\newcommand{\setsection}[2] {
\setcounter{section}{#1}
\setcounter{subsection}{0}
\setcounter{equation}{0}
\setcounter{conjecture}{0}
\setcounter{assumption}{0}
\setcounter{question}{0}
\setcounter{definition}{0}
\setcounter{theorem}{0}
\setcounter{corollary}{0}
\setcounter{lemma}{0}
\setcounter{proposition}{0}
\setcounter{remark}{0}
\setcounter{appen}{0}
\setsection*{\large \bf \thesection. #2}}

\newenvironment{mylist}[1]{\begin{list}{}
{\setlength{\itemindent}{#1mm}}
{\setlength{\itemsep}{0ex plus 0.2ex}}
{\setlength{\parsep}{0.5ex plus 0.2ex}}
{\setlength{\labelwidth}{10mm}}
}{\end{list}}

\newcommand{\setnewcounter} {
\setcounter{subsection}{0}
\setcounter{equation}{0}
\setcounter{conjecture}{0}
\setcounter{assumption}{0}
\setcounter{question}{0}
\setcounter{definition}{0}
\setcounter{theorem}{0}
\setcounter{corollary}{0}
\setcounter{lemma}{0}
\setcounter{proposition}{0}
\setcounter{remark}{0}
}

\begin{document}

\title{\bf \Large  Tight matrices and heavy traffic steady state convergence in queueing networks}

\author{J.G. Dai\\Cornell University\\ \and Yiquan Ji\\Tsinghua University \and Masakiyo Miyazawa\\Tokyo University of Science\\}
\date{June 29, 2025}

\maketitle

\begin{abstract}
We are interested to prove that the stationary distribution of a multiclass queueing network converges to the stationary distribution of a semimartingale reflecting Brownian motion (SRBM) in heavy traffic.  A key condition for this convergence is that the sequence of the pre-limit stationary distributions under appropriate scaling is tight. In \citet{BravDaiMiya2025}, a sufficient condition for this tightness is introduced in the term of the reflection matrix $R$ of the SRBM, which is coined for $R$ to be ``tight''. In this paper, we study how we can verify this tightness of $R$ of an SRBM. For a $2$-dimensional SRBM, we give necessary and sufficient conditions for $R$ to be tight, while, for a general dimension, we only give sufficient conditions. We then apply these results to the SRBMs arising from the diffusion approximations of  multiclass queueing networks with static buffer priority service disciplines that are studied in \cite{BravDaiMiya2025}. It is shown that $R$ is always tight for this network with two stations if $R$ is completely-$\sr{S}$. For the case of more than two stations, it is shown that $R$ is tight for reentrant lines with last-buffer-first-service (LBFS) discipline, but it is not always tight for reentrant line with  first-buffer-first-service (FBFS) discipline.
\end{abstract}

 \begin{quotation}

\end{quotation}

\section{Introduction}
\label{sec:introduction}
\setnewcounter

We are interested to prove that the stationary distribution of a multiclass queueing network converges to the stationary distribution of a semimartingale reflecting Brownian motion (SRBM) in heavy traffic, or as the load at each service station
becomes ``critical.''  A key condition for this convergence is that the sequence of the pre-limit stationary distributions under appropriate scaling is tight. In \citet{BravDaiMiya2025}, a sufficient condition for this tightness is introduced in the term of the reflection matrix $R$ of the SRBM, which is coined for $R$ to be ``tight''. This tightness of $R$ is applicable to any queueing network for proving that the heavy traffic limit of its stationary distribution is identical with the stationary distribution of an SRBM. Here, it is not necessary to show for this SRBM to be the process limit of the queueing process of the network in heavy traffic. Instead of it, we need to show that the stationary equation, called the basic adjoint relationship (BAR), for the pre-limit stationary distribution converges to an equation similar to the BAR of the SRBM. This part has been studied in the literature (see \cite{BravDaiMiya2017,BravDaiMiya2025}), and we do not consider it here.

In this paper, we study how we can verify this tightness of $R$ for an SRBM. For a $2$-dimensional SRBM, we give necessary and sufficient conditions for $R$ to be tight, while, for a general dimension, we only give sufficient conditions. We then apply these results to the SRBMs arising from the diffusion approximations of  multiclass queueing networks with static buffer priority service disciplines that are studied in \cite{BravDaiMiya2025}. It is shown that $R$ is always tight for this network with two stations if $R$ is completely-$\sr{S}$. For the case of more than two stations, it is shown that $R$ is tight for reentrant lines with last-buffer-first-service (LBFS) discipline, but we construct an example of FBFS
reentrant queueing network whose reflection matrix is not tight.
This example leaves the problem open if the steady-state convergence
holds for FBFS reentrant queueing networks in heavy traffic. For literature review of Brownian approximations of queueing networks, see for example, the introduction of \citet{BravDaiMiya2025}.

This paper consists of five sections and an appendix. In \sectn{problem}, we discuss about what problem motivates this study. We then introduce the tight matrix  and main results
(Theorems \thmt{1} and \thmt{2}) in \sectn{main}. In \sectn{application}, we apply main results to SRBMs arising from  multiclass queueing networks with static buffer priority (SBP) service disciplines.
Main theorems are proved in \sectn{proof-theorem}. 

\section{Problem of our interest}
\label{sec:problem}
\setnewcounter

Consider a multi-class queueing network with $J$ stations for service and $K$ classes of customers such that $0 < J \le K < \infty$,
where customers who are served in different stations are classified into different classes. We are interested in the diffusion approximation of this queueing network in heavy traffic. This topic was introduced in \citet{Harr1988} and subsequently studied by many authors; see, for example, \cite{HarrNguy1993,Bram1998,Will1998,ChenZhan2000b,BramDai2001}. To study it, we consider the sequence of these networks indexed by $n=1,2,\ldots$. Let $L^{(n)}_{k}(t)$ be the number of class $k$ customers in the $n$-th network at time $t \ge 0$, which is referred to as the queue size of class $k$ customer at time $t$. For the diffusion approximation, define
\begin{align}
\label{eq:hQ-t}
   \widehat{L}^{(n)}_{k}(t) = n^{-1/2} L^{(n)}_{k}(nt), \qquad t \ge 0, k =1,2,\ldots,K,
\end{align}
and let $\widehat{L}^{(n)}(t) = (\widehat{L}^{(n)}_{1}(t),\widehat{L}^{(n)}_{2}(t),\ldots,\widehat{L}^{(n)}_{K}(t))$. Then, we refer to the $K$-dimensional process $\widehat{L}^{(n)}(\cdot) \equiv \{\widehat{L}^{(n)}(t); t \ge 0\}$ as a diffusion scaled queueing process.

We consider the situation that $\widehat{L}^{(n)}_{k}(t)$ vanishes in probability as $n \to \infty$ in heavy traffic for a subset of customer classes $k$. This phenomenon of vanishing is called a state space collapse. Assume state space collapse has been proved. We then focus on the limiting behaviors of the diffusion scaled queues in remaining classes, which do not  vanish in heavy traffic. Let $d$ be their total number. By  appropriately renumbering customer classes, we denote the indexes of those queues by $1,2,\ldots,d$. Define
\begin{align*}
  \widehat{L}^{(n)}_{[1,d]}(t) = (\widehat{L}^{(n)}_{1}(t), \widehat{L}^{(n)}_{2}(t), \ldots, \widehat{L}^{(n)}_{d}(t)),
\end{align*}
and let $\widehat{L}^{(n)}_{[1,d]}(\cdot) = \{\widehat{L}^{(n)}_{[1,d]}(t); t \ge 0\}$. Note that $J=d$ in our applications, which means that all the stations are in heavy traffic as $n$ gets large, but $d < J$ may occur in general. For convenience, let $\sr{N}_{d} = \{1,2,\ldots,d\}$.  We denote by $\dd{R}^{d}$ the $d$-dimensional real vector space, let $\dd{R}^{d}_{+} = \{x \in \dd{R}^{d}; x \ge 0\}$, and let $\dd{R}^{d}_{-} = \{x \in \dd{R}^{d}; x \le 0\}$, where all vector inequalities are interpreted componentwise.

Thus, we are interested in the limiting behavior of the sequence of $\widehat{L}^{(n)}_{[1,d]}(\cdot)$ for $n \ge 1$ in heavy traffic. There are two different modes for this limit. One is a process limit, which is the weak limit of process $\widehat{L}^{(n)}_{[1,d]}(\cdot)$ as $n \to \infty$. Another is the limit of the stationary distribution of $\widehat{L}^{(n)}_{[1,d]}(\cdot)$ under the assumption that $\widehat{L}^{(n)}_{[1,d]}(\cdot)$ is stable for each $n \ge 1$, where a stochastic process is called stable if it has a unique stationary distribution. We denote a random vector subject to this stationary distribution by
\begin{align*}
  \widehat{L}^{(n)}_{[1,d]} \equiv (\widehat{L}^{(n)}_{1}, \widehat{L}^{(n)}_{2}, \ldots, \widehat{L}^{(n)}_{d}).
\end{align*}

Both modes of the limit are important in application and have been independently studied in the literature because their derivations require different techniques. We are particularly interested in the distributional limit of $\widehat{L}^{(n)}_{[1,d]}$ as $n \to \infty$ under the stability condition. Denote the distribution of $\widehat{L}^{(n)}_{[1,d]}$ by $\widehat{\nu}^{(n)}$, then our interest is to find a unique probability distribution $\widehat{\nu}$ on $\dd{R}_{+}^{d}$ such that
\begin{align}
\label{eq:limit-hnu}
  \widehat{\nu}^{(n)} \os{d}{\Rightarrow} \widehat{\nu}, \qquad n \to \infty.
\end{align}
where ``$\os{d}{\Rightarrow}$'' stands for the weak convergence of probability distributions. To consider this problem, note that, by the vague sequential compactness of $\{\widehat{\nu}^{(n)}; n \ge 1\}$ (e.g., see Theorem 5.19 of \cite{Kall2001}), there is a subprobability measure $\nu$ on $\dd{R}_{+}^{d}$ and a subsequence $\{\widehat{\nu}^{(n_{\ell})}; \ell \ge 1\}$ for any subsequence of $\{\widehat{\nu}^{(n)}; n \ge 1\}$ such that
\begin{align}
\label{eq:limit-nu}
  \widehat{\nu}^{(n_{\ell})} \os{v}{\Rightarrow} \nu, \qquad \ell \to \infty.
\end{align}
where ``$\os{v}{\Rightarrow}$'' stands for the vague convergence of measure. Then, the problem is to show that $\nu$ is a probability measure and uniquely obtained independent of the subsequence $\{\widehat{\nu}^{(n_{\ell})}; \ell \ge 1\}$. Namely, \eq{limit-nu} implies \eq{limit-hnu} with $\widehat{\nu} = \nu$ if the following two conditions hold.
\begin{mylist}{3}
\item [(\sect{problem}.a)] $\{\widehat{\nu}^{(n)}; n \ge 1\}$ is tight. That is, there is a compact set $C \subset \dd{R}_{+}^{d}$ for any $\varepsilon > 0$ such that $\inf_{n \ge 1} \widehat{\nu}^{(n)}(C) \ge 1 - \varepsilon$.
\item [(\sect{problem}.b)] The $\nu$ of \eq{limit-nu} is unique and independent of the choice of the subsequences. 
\end{mylist}

Thus, we aim to find tractable conditions for these two conditions to hold. To discuss them as well as to present our results, we review the various
matrix classes that are  relevant to the Brownian approximations
of queueing networks.

A $d\times d$ square matrix $A$ is called a completely-$\sr{S}$-matrix
if every principal submatrix $C$ of $A$ has a positive vector $x$
of proper dimension such that $C x > 0$. $A$ is called a $\sr{P}$ matrix if all its
principal minors are positive, and called a $\sr{M}$-matrix if it is a
$\sr{P}$-matrix and all its off diagonal entries are not positive. The
class of $\sr{M}$-matrices is a subset of  $\sr{P}$-matrices,
which is a subset of the class of completely-$\sr{S}$-matrices. A
$d\times d$ matrix $A$ (not necessarily symmetric) is called a positive definite matrix if
$x' A x > 0$ for all non-null $d$-dimensional column real vectors $x$,
and $x'$ is the transpose of $x$.  A positive definite matrix is a
$\sr{P}$-matrix, and the inverse of a $\sr{P}$-matrix is a
$\sr{P}$-matrix.  See, e.g. \cite{BermPlem1994}. We denote the
$(i,j)$-entry of matrix $A$ by $A_{i,j}$ unless stated
otherwise. Similarly, the $i$-th entry of vector $x$ is denoted by
$x_{i}$.

There are two approaches for proving (\sect{problem}.a) and (\sect{problem}.b). One is to use the process limit when it is available. This approach, known as ``interchange of limits'',  has been adopted by the pioneering paper of \cite{GamaZeev2006} and many subsequent works (e.g, see \cite{BudhLee2009}). In this approach, it is typical that $\widehat{L}^{(n)}_{[1,d]}(\cdot)$ weakly converges to an SRBM with drift $\beta$ and reflection matrix $R$, which is assumed to be a complete-$\sr{S}$ matrix. Namely, denote this SRBM by $Z(\cdot)$, then it is the unique solution of the stochastic equation:
\begin{align}
\label{eq:SRBM-Z}
  Z(t) = Z(0) + X(t) + R Y(t) \ge 0, \qquad t \ge 0,
\end{align}
where $R$ is a $d \times d$ matrix, called a reflection matrix,
$X(\cdot)$ is the $d$-dimensional Brownian motion with drift vector
$\beta \in \dd{R}^{d}$ and covariance matrix $\Sigma$ whose $(i,j)$
entry is $\Sigma_{i,j}$ for $i,j =1,2,\ldots,d$, and
$Y(t) \equiv (Y_{1}(t), Y_{2}(t), \ldots, Y_{d}(t))$ in which
$Y_{i}(t)$ is non-decreasing and $\int_{0}^{t} Z_{i}(s) dY_{i}(s) = 0$
for $t \ge 0$ and $i\in \sr{N}_{d}$. The process
$Y(\cdot) = \{Y(t); t \ge 0\}$ is called a regulator, which keeps
$Z(t)$ to be nonnegative. For generalized Jackson networks,
  authors in \cite{GamaZeev2006} and \cite{BudhLee2009}) utilize the
  Skorohod map introduced in \cite{HarrReim1981} to prove (\sect{problem}.a).

Once (\sect{problem}.a) is verified, it follows from the process limit that $Z(\cdot)$ is stable and $\nu$ is the stationary distribution of $Z(\cdot)$ as argued below Theorem 3.2 of \cite{BudhLee2009}, while this stationary distribution is unique by Theorem (1) of \cite{HarrWill1987}. Thus, (\sect{problem}.a) together with the process limit proves (\sect{problem}.b).

Another approach to prove (\sect{problem}.a) and (\sect{problem}.b) is to use the stationary equation of $\widehat{L}^{(n)}_{[1,d]}(\cdot)$ for $n \ge 1$, which is called a basic adjoint relationship, BAR for short, in \cite{BravDaiMiya2025}.
Following \cite{BravDaiMiya2025},
we refer to this approach as BAR approach. In this approach, we define the moment generating functions $\varphi^{(n)}$ as
\begin{align*}
  & \varphi^{(n)}(\theta) = \dd{E}\left[e^{\br{\theta,\widehat{L}^{(n)}}}\right]=\int_{\R^d_+} e^{\langle \theta, x\rangle} \widehat{\nu}^{(n)}(dx)
      \qquad  \theta \in \dd{R}_{-}^{d},
\end{align*}
the  corresponding ``boundary'' moment generating functions
  \begin{align*}
\varphi^{(n)}_j(\theta) = \int_{\R^d_+} e^{\langle \theta, x\rangle} \widehat{\nu}^{(n)}_j(dx)
  \end{align*}
  for $j\in {\cal{N}}_d$ and $n \ge 1$, where
  $\widehat{\nu}^{(n)}_{j}$ is some probability measure on
  $\dd{R}_{+}^{d}$ supported by
  $\dd{R}_{+}^{d}|_{j=0} \equiv \{x \equiv (x_{1}, x_{2}, \ldots,
  x_{d}) \in \dd{R}_{+}^{d}; x_{j} = 0\}$. By the vague sequential
compactness of
\begin{align*}
  \Big\{\big(\widehat{\nu}^{(n)}, \widehat{\nu}^{(n)}_j, j\in {\cal N}_d\big); n \ge 1\Big\},
\end{align*}
for any sequence $\{n_{\ell}; \ell\ge 1\}\subset \{1, 2, \ldots \}$, 
there is a subsequence $\{n'_{\ell}; \ell \ge 1\}\subset \{n_\ell;\ell \ge 1\}$ such that
\begin{align}
\label{eq:mgf-phi}
 \lim_{\ell \to \infty} \varphi^{(n'_{\ell})}(\theta) \quad \text{ and } \quad  \lim_{\ell \to \infty} \varphi^{(n'_{\ell})}_{j}(\theta), 
\end{align}
exist for each $\theta \in \dd{R}_{-}^{d}$. These limits are 
denoted by $\varphi(\theta)$ and $\varphi_j(\theta)$, respectively.
Then, the BAR approach requires to prove the following condition:

\begin{mylist}{3}
\item [(\sect{problem}.c)]
$\varphi(\theta)$ and $\varphi_j(\theta)$  
satisfy the following equation
\begin{align}
\label{eq:limit-BAR}
  - \br{Rb,\theta} \varphi(\theta) + \frac 12 \br{\Sigma \theta, \Sigma \theta} \varphi(\theta) + \sum_{i=1}^{d} \theta_{i} \sum_{j=1}^{d} R_{i,j} b_{j} \varphi_{j}(\theta) = 0, \qquad \theta \in \dd{R}_{-}^{d}
\end{align}
for a $d$-dimensional positive vector $b$ and $d \times d$ matrices $\Sigma$ and $R$ such that $R$ is complete-$\sr{S}$ and $\Sigma$ is 
symmetric and positive definite. 
\end{mylist}

Condition (\sect{problem}.c) is verified for the generalized Jackson
network and a multiclass queueing network under some extra conditions
in \cite{BravDaiMiya2017} and \cite{BravDaiMiya2025}, respectively. In
\cite{BravDaiMiya2017}, $\widehat{\nu}^{(n)}_{j}$ is defined as the
conditional distribution of $\widehat{L}^{(n)}_{[1,d]}$ given
$\widehat{L}^{(n)}_{j} = 0$. In \cite{BravDaiMiya2025},
$\varphi^{(n)}_j(\theta)$ is defined through (6.39) (using $\phi^{(r)}_{k}$
following the notional system there) and the corresponding
probability measure $\nu^{(n)}_j$ is supported by
$\dd{R}_{+}^{d}|_{j=0}$.

Note that $\varphi$ and $\varphi_{j}$ may not be the moment generating functions of probability distributions and may depend on the subsequence $\{n'_{\ell}; \ell \ge 1\}$. Hence, in the BAR approach, it is required to show that $\varphi$ and $\varphi_{j}$ are the moment generating functions of probability distributions which are unique and independent of the choice of the subsequence $\{n'_{\ell}; \ell \ge 1\}$. For this, we introduce the following notations. For $c \equiv (c_{1},c_{2}, \ldots, c_{d}) > 0$ and $D \subset \sr{N}_{d}$, let $c_{D}$ be the vector in $\dd{R}_{+}^{d}$ such that $[c_{D}]_{i} = c_{i} 1(i \in D)$ for $i \in \sr{N}_{d}$. Then, for negative number $\theta_{0}$, if the limit of $\varphi(\theta_{0} c_{D})$ exists and is independent of $c > 0$ as $\theta_{0} \uparrow 0$ for any $c > 0$,  we denote this limit by $\varphi(-0_{D})$. In particular, $\varphi(-0_{\sr{N}_{d}})$ is simply denoted by $\varphi(-0)$. We similarly define $\varphi_{j}(-0_{D})$ and $\varphi_{j}(-0)$. We are ready to present the following lemma.

\begin{lemma}
  \label{lem:limit-tight}
  Let $\varphi(\theta)$ and $\varphi_j(\theta)$, $j\in {\cal N}_d$, be limits in \eq{mgf-phi}  along any convergent subsequence $\{n'_\ell; \ell\ge 1\}$.
  Assume (\sect{problem}.c) holds and  assume further  $\varphi(0-) = \varphi_{j}(0-) = 1$ for all $j \in \sr{N}_{d}$. Then, (\sect{problem}.a) and (\sect{problem}.b) hold. Furthermore, define SRBM $Z(\cdot)$ with drift $\beta \equiv - Rb$, covariance matrix $\Sigma$ and reflection matrix $R$. Then $Z(\cdot)$ has a unique stationary distribution.
\end{lemma}
\begin{proof}
Since $\varphi(0-) = \varphi_{j}(0-) = 1$ for all
  $j \in \sr{N}_{d}$, it follows that  $\varphi(\theta)$ and $\varphi_j(\theta)$ are
  moment generating functions of some probability measures $\nu$ and
  $\nu_j$, respectively, and $\nu_j$ has support on
  $\dd{R}_{+}^{d}|_{j=0}$.  Because $R$ is assumed to be
  completely-${\cal S}$ and $\Sigma$ is assumed to be symmetric and
  positive definite, it follows from \cite{TaylWill1993} that an SRBM
  $Z(\cdot)$ is uniquely defined (in distribution) for each initial
  distribution with drift $\beta = - Rb$, covariance $\Sigma$ and
  reflection matrix $R$.  From (\sect{problem}.c),
  $(\nu, \nu_1, \ldots, \ldots, \nu_d)$ satisfies a basic adjoint
  adjoint relationship (BAR) in Theorem 8.1 of \cite{HarrWill1987},
  which is equivalent to the moment generating function version
  \eq{limit-BAR}; see Appendix D in the arXiv version of
  \cite{DaiMiyaWu2014} for the equivalence.  It follows from
  \citet{DaiKurt1994} that $\nu$ is a stationary distribution of SRBM
  $Z$. The stationary distribution $\nu$ is unique following
  the  argument in Section 7 of 
  \cite{HarrWill1987}. (In \cite{HarrWill1987}, the reflection matrix
  is assumed to be an ${\cal M}$-matrix. However, its uniqueness
  argument is  general, applicable to completely-${\cal S}$ matrix as well; see also the
  comment four lines below (2.5) in \cite{DaiDiek2011}). Thus,
  (\sect{problem}.a) and (\sect{problem}.b) follow from the fact that
  $\nu$ is a probability measure and $\nu$ is the
  unique stationary distribution of the SRBM.
\end{proof}

In applications, the invertibility of the reflection matrix $R$ is usually required to define it. In this case, the assumptions in \lem{limit-tight} can be weakened as shown below.

\begin{corollary}\rm
  \label{cor:limit-tight}
  Let $\varphi(\theta)$ and $\varphi_j(\theta)$, $j\in {\cal N}_d$, be
  limits in \eq{mgf-phi} along any convergent subsequence
  $\{n'_\ell; \ell\ge 1\}$.  Assume (\sect{problem}.c) and that $R$ is
  invertible.  Then, (i) $\varphi(0-) = \varphi_{i}(0-)$ for all
  $i \in \sr{N}_{d}$.  As a consequence, (ii) $\varphi(0-) = 1$ is
  necessary and sufficient condition for (\sect{problem}.a) and
  (\sect{problem}.b) to hold.
\end{corollary}

\begin{proof}
  Fix the subsequence $\{n'_{\ell}; \ell \ge 1\}$ so that \eq{mgf-phi}
  holds. We first prove (i). Let $c = (c_{1},c_{2}, \ldots, c_{d}) \in \dd{R}_{+}^{d}$ and let $\theta_{0} < 0$. Then, we substitute $\theta = \theta_{0} c$ into \eq{limit-BAR}, divide by $\theta_{0}$, and obtain 
\begin{align}
\label{eq:Rb-c}
 \br{Rb,c} \varphi(\theta_{0} c) = \frac 12 \theta_{0} \br{\Sigma c, \Sigma c} \varphi(\theta_{0} c) + \sum_{i=1}^{d} c_{i} \sum_{j=1}^{d} R_{i,j} b_{j} \varphi_{j}(\theta_{0} c), \qquad \theta_{0} < 0, \; c \ge 0.
\end{align}
In \eq{Rb-c}, first let $c_{i} \downarrow 0$ for $i \not= k$ while $c_{k} >0$ is unchanged, then let $\theta_{0} \uparrow 0$, which yields
\begin{align} 
\label{eq:Rb-ci}
 \sum_{j=1}^{d} R_{k,j} b_{j} \varphi(0-) = \sum_{j=1}^{d} R_{k,j} b_{j} \varphi_{j}(0-), \qquad k \in \sr{N}_{d}.
\end{align}
Since $R$ is invertible, $R$ has the inverse $R^{-1}$. Multiply both sides of \eq{Rb-ci} by $[R^{-1}]_{i,k}$ from the left and sum up for all $k \in \sr{N}_{d}$, then $b_{i} \varphi(0-) = b_{i} \varphi_{i}(0-)$ for $i \in \sr{N}_{d}$. Thus,
(i) is proved because $b > 0$ by (\sect{problem}.c). Now we prove (ii). Assume $\varphi(0-) = 1$. By (i),
$\varphi(0-) = \varphi_{j}(0-) = 1$ for all $j \in \sr{N}_{d}$. Hence,
(\sect{problem}.a) and (\sect{problem}.b) hold by
\lem{limit-tight}. Thus, the sufficiency of $\varphi(0-) = 1$ is
proved. It is clear that  $\varphi(0-) = 1$ is
a necessary condition. Thus, (ii) is proved.
\end{proof}


\section{Tight matrices and the main results}
\label{sec:main}
\setnewcounter

In this paper, we will seek verifiable conditions for \eq{limit-hnu}
to hold and for $\widehat{\nu}$ to be identified. One component of the BAR approach
is to  prove (\sect{problem}.c). By \lem{limit-tight}, the remaining component of the BAR approach is to prove $\varphi(0-) = \varphi_{j}(0-) = 1$, $j \in \sr{N}_{d}$.
To prove the latter component of the BAR approach,  we take the  following definition from \citet{BravDaiMiya2025}.

\begin{definition}[Tight matrix]
\label{dfn:tight}
Fix a positive integer $d$. Let $b$ be a $d$-dimensional positive vector and let $R$ be a $d \times d$ completely-$\sr{S}$-matrix. Then, $(R,b)$ is called a \emph{tight system} if the set of variables $\{x_{D} \in [0,1]; D \subset \sr{N}_{d}\}$ and $\{x^{(j)}_{D}; D \subset \sr{N}_{d}\}$ for $j \in \sr{N}_{d}$ satisfying the following set of linear equations and inequalities:
\begin{align}  
\label{eq:tight-c1}
& \sum_{j\in\sr{N}_{d}} R_{ij} b_{j} (x_{D}^{(j)} - x_{D}) = 0,  \qquad i\in D \subset \sr{N}_{d},\\
\label{eq:tight-c2}
&x_{D} \geq x_{D^{'}}, \qquad x_{D}^{(j)} \geq x_{D^{'}}^{(j)}, \qquad j\in \sr{N}_{d}, \quad D \subset D^{'} \subset \sr{N}_{d}, \\
\label{eq:tight-c3} 
&x_{D}^{(j)} = x_{D \backslash \{j\}}^{(j)}, \quad j\in D \subset \sr{N}_{d}, \\ & x_{\emptyset} = x_{\emptyset}^{(j)} = 1,  \quad   j\in \sr{N}_{d}
  \label{eq:tight-c4}
\end{align}
has the unique solution:
\begin{align}
\label{eq:xD-s}
  x_{D} = x^{(j)}_{D} = 1, \qquad D \subset \sr{N}_{d}, \quad j \in \sr{N}_{d}.
\end{align}
In particular, if $(R,b)$ is a tight system for all $b > 0$, then $R$ is called a \emph{tight matrix}. 
\end{definition}

\begin{remark}
\label{rem:tight}
The condition \eq{tight-c1} comes from \eq{Rb-c}. To see this, substitute $c = c_{D}$ into \eq{Rb-c} and let $\theta_{0} \uparrow 0$. Then we have
\begin{align*}
   \sum_{k \in D} \sum_{j=1}^{d} c_{k} R_{k,j} b_{j} \varphi(-0_{D}) = \sum_{k \in D} \sum_{j=1}^{d} c_{k} R_{k,j} b_{j} \varphi_{j}(-0_{D}), \qquad D \subset \sr{N}_{d}.
\end{align*}
Hence, if we put $x_{D} = \varphi(-0_{D})$ and $x^{(j)}_{D} = \varphi_{j}(-0_{D})$, then we have \eq{tight-c1}. The remaining conditions \eq{tight-c2}--\eq{xD-s} obviously follow from this setting of $x_{D}$ and $x^{(j)}_{D}$.
\end{remark}

The significance of introducing a tight matrix is demonstrated in Lemma 4.2 of \cite{BravDaiMiya2025}. Namely, by this lemma and \lem{limit-tight}, we have the following fact.
\begin{lemma}\rm
\label{lem:tight matrix}
Assume that (\sect{problem}.c) hold. If $(R,b)$ is a tight system, then (\sect{problem}.a) and (\sect{problem}.b) hold, and therefore we have \eq{limit-hnu}.
\end{lemma}

It was proved in \cite{BravDaiMiya2017} that each $\sr{M}$-matrix is a tight matrix.
In this paper, 
we present  two theorems for verifying the tightness of $R$. The first theorem is limited to the
case when $d=2$, but the conditions are necessary and sufficient for
$R$ to be tight. The second theorem is for the case when $d$ is general,
but the conditions are sufficient only. These theorems are proved in
\sectn{proof-theorem}.

\begin{theorem}
\label{thm:1}
For $d=2$, assume
\begin{mylist}{0}
\item [(\sect{main}.a)] $R_{1,1} > 0$ and $R_{2,2} > 0$,
\end{mylist}
and define the following conditions on $R$.
\begin{mylist}{0}
\item [(\sect{main}.b)] $R_{1,2} \le 0$, $R_{2,1} \le 0$ and $R_{1,1} R_{2,2} - R_{1,2} R_{2,1} > 0$,
\item [(\sect{main}.c)] $R_{1,2} < 0, R_{2,1} > 0$ or $R_{1,2} > 0, R_{2,1} < 0$,
\item [(\sect{main}.d)] $R_{1,2}>0$, $R_{2,1} > 0$, or $R_{1,2} = 0, R_{2,1} > 0$, or $R_{2,1} = 0, R_{1,2} > 0$,
\item [(\sect{main}.e)] $R_{1,2} \le 0$, $R_{2,1} \le 0$ and $R_{1,1} R_{2,2} - R_{1,2} R_{2,1} \le 0$.
\end{mylist}
Then, the conditions (\sect{main}.b)--(\sect{main}.e) are mutually exclusive, and $R$ is a tight completely-$\sr{S}$-matrix if and only if either (\sect{main}.b) or (\sect{main}.c) hold. If the conditions (\sect{main}.d) holds, $R$ is completely-$\sr{S}$ but not tight. If the conditions (\sect{main}.e) holds, $R$ is not completely-$\sr{S}$. 
\end{theorem}

Note that $R$ is an $\sr{M}$-matrix if and only if the conditions (\sect{main}.a) and (\sect{main}.b) hold, while $R$ is a $\sr{P}$-matrix if and only if (\sect{main}.a) and $R_{1,1} R_{2,2} - R_{1,2} R_{2,1} > 0$ hold. Hence, an $\sr{M}$-matrix $R$ is tight, but not all $\sr{P}$-matrix $R$ is tight because (\sect{main}.d) may hold for a $\sr{P}$-matrix $R$.

\begin{theorem}\rm
\label{thm:2}
Assume that  $d \times d$ matrix $R$ is a $\sr{P}$-matrix that satisfies the following conditions
\begin{align}
\label{eq:R1}
 & R_{i,j} = \left\{
\begin{array}{ll}
 \mbox{a real number}, & j \ge i+1, \; 2  \le i \le d-1,\\
 \mbox{positive}, \quad & j=i,  \\
 \mbox{negative}, \quad & j = i-1, \; 2  \le i \le d,\\
 0, \quad & j \le  i-2, \; 3  \le i \le d.
\end{array}
\right.
\end{align}
Then $R$ is a tight matrix.
\end{theorem}
As an illustration for \thm{2},  a $4\times 4$ $\sr{P}$-matrix of the form
\begin{align*}
  R =
  \begin{pmatrix}
    + & \times & \times & \times\\
    - & + & \times & \times\\
    0 & - & + & \times\\
    0 & 0 & - & +
  \end{pmatrix}
\end{align*}
is always a tight matrix. Note that the \thm{2} is a special case of \thm{1}
for $d=2$. \thm{2} will be used in \sectn{application} to prove the tightness of a reentrant
line with the LBFS discipline. We also note that both theorems only
concern with signs of the entries of $R$, and therefore the values of
positive components of $b$ have no role in verifying $(R,b)$ to be a tight
system.

\section{Applications to multiclass networks}
\label{sec:application}
\setnewcounter

Theorems \thmt{1} and \thmt{2} are applicable to the reflecting matrix $R$ of any SRBM, but there may be difficulty in their applications because the reflecting matrix $R$ may not be obtained in tractable form for the SRBM arising from  a queueing network. This is particularly the case for a multiclass queueing network. We attack this problem for multiclass networks with SBP service disciplines. We first introduce this network, and derives new expression for $R$, then focus on a reentrant line as its special case.

\subsection{Multiclass networks with SBP disciplines}
\label{sec:general}

Multiclass queueing networks have been extensively studied in
literature. See, for example, \cite{BravDaiMiya2025}. Here we give a
brief introduction, presenting the minimal parameters and notations
that are necessary for this paper.  Consider a multiclass queueing
network that has $d$ single-server stations and $K$ customer classes for
positive integers $d$ and $K > d$. Denote the sets of all stations and
all customer classes respectively by
\begin{align*}
  S =\{1,2,\ldots,d\}, \qquad \sr{K} = \{1, 2, \ldots , K\}.
\end{align*}
A customer arrives from outside of the network, and visits a finite number
of stations to get service sequentially, then leaves the network. Each
customer at a station is uniquely assigned a customer class. Denote the
station at which class $k \in \sr{K}$ customer belongs to by $s(k)\in S$. For
station $i \in S$, define $\sr{C}_{i} = \{k \in \sr{K}; s(k) = i\}$,
which is the set of all customer classes served at station $i$.

Since each station has a single server and may have customers of multiple classes, we need a service discipline that
dictates the order at which customers are processed. For this, we take a static buffer priority (SBP) service discipline. Under this service discipline, each customer class has its own unlimited buffer at the station for holding waiting customers, and is assigned a unique priority. Each customer is served in preemptive resume manner according to its priority assigned through its class.

\begin{definition}[Priority]
  \label{dfn:priority-class}
  \emph{An SBP service discipline $p$} is a one-to-one mapping
  $p: \sr{K}\to \sr{K}$. For each class $k\in \sr{K}$, $p(k)$ is the
  priority level of class $k$. At each station $i\in S$, class $k\in \sr{C}_i$
  has a higher priority than class $k'\in \sr{C}_i$ if and only if
  $p(k) < p(k')$.
\end{definition}

We next introduce some sets of customer classes according to their priorities.
\begin{definition}[Classes $\sr{L}$ and $\sr{H}$]
\label{dfn:class-L}
For each station $i\in S$, we use $\ell(i)\in \sr{C}_i$ to denote the lowest priority class at the station.
By re-labeling station if necessary, we assume $\ell(i) < \ell(j)$  for $i < j$.
Define $\sr{L} = \{\ell(1), \ell(2),\ldots , \ell(d)\} \subset \sr{K}$.
$\sr{L}$ is called the set of the low priority classes. Define $  \sr{H} = \sr{K} \setminus \sr{L}$,
which is called the set of high priority classes. 
\end{definition}

Let $\lambda_{k} \ge 0$ be the arrival rate of class-$k$ customers from the outside, and denote the vector whose $k$-th entry is $\lambda_{k}$ by $\lambda$. When class $k$ customer completes its service at station $s(k)$, the customer either changes its class from $k$ to some class $k' \in \sr{K}$ with probability $P_{k,k'}$ independently of everything else or leaves the network. The $K \times K$ matrix $P \equiv \{P_{k.k'}; k.k' \in \sr{K}\}$ is called a routing probability matrix. It is assumed that $I - P$ is invertible. This is equivalent that each customer is assumed to eventually leave the network. Let $m_{k}$ be the mean service time for class $k$ customers. Let $M$ be the $K \times K$ diagonal matrix whose $i$-th diagonal entry is $m_{i}$, namely,
\begin{align*}
  M = {\rm diag}(m_{1}, m_{2}, \ldots , m_{K}).
\end{align*}

We now define key matrices for the multiclass queueing with SBP service discipline.  
Fix a station $i\in S$. For class $k \in \sr{C}_i$, let $H(k)$ be the set of classes at station $i$ whose priorities are at least as high as class $k$. Namely,
\begin{align*}
  H(k) = \{k' \in \sr{C}_{i}: p(k') \le p(k)\}.
\end{align*}
Obviously, $H(\ell(i)) = \sr{C}_{i}$ for each station $i \in S$.  Let
$H_{+}(k) = H(k) \backslash \{k\}$, which is the set of classes at
station $s(k)$ whose priorities are strictly higher than class
$k$. For each $k \in \sr{H}$, define $\mbox{$k+$}$ to be the lowest
class in $H_{+}(k)$.  The set $H_{+}(k)$ is empty if $k$ is the
highest priority class at a station. In the latter case, $\mbox{$k+$}$
is undefined.

Following \cite{BravDaiMiya2025}, we first define the $K\times K$ matrix $A$.
\begin{definition}
\label{dfn:A}
Define $K \times K$ matrix $A$ as
\begin{align}
\label{eq:A}
  A = (I - P^{'})M^{-1}(I - B),
\end{align}
where $P^{'}$ is the transpose of the routing matrix $P$, and $B$ is the $K \times K$ matrix defined by
\begin{align}
\label{eq:B}
  B_{k,k'} = \left\{
\begin{array}{ll}
 1, \;\; &  k'= \mbox{$k+$}, \\
 0, & \mbox{otherwise},
\end{array}
\right.
\end{align}
\end{definition}
In \cite{BravDaiMiya2025}, the authors define matrices $A_{L}$, $A_{LH}$, $A_{HL}$ and $A_{H}$ as
\begin{mylist}{3}
\item [(i)] $A_{L}$ and $A_{H}$ are the principal submatrices corresponding the index set $\mathcal{L}\subset \sr{K}$ and
$\mathcal{H}\subset \sr{K}$, respectively.
\item [(ii)] $A_{LH}$ and $A_{HL}$ are corresponding other blocks of $A$.
\end{mylist}
Thus, if the indexes of customer classes are appropriately chosen, then we can write $A$ as
\begin{align}\label{eq:A-b}
  A =
  \begin{pmatrix}
    A_{L} & A_{LH} \\
    A_{HL} & A_{H}
  \end{pmatrix}.
\end{align}
  Assume the matrix $A_{H}$ is invertible. The reflection matrix  $\tilde{R}$, as an $\sr{L} \times \sr{L}$ indexed matrix, is given by 
  \begin{align}
  \label{eq:R}
  \tilde{R} = A_{L} - A_{LH} A_{H}^{-1} A_{HL}.
  \end{align}
  Recall the definition of $\ell: S\to \sr{K}$, where $\ell(i)$ is the
  lowest priority class at station $i$. This map is a one-to-one map
  from $S\to \sr{L}\subset\sr{K}$. It is convenient to work with a
  $S\times S$ indexed reflection matrix $R$. Their relationship is
  given by $R_{i,j}=\tilde{R}_{\ell(i), \ell(j)}$ for $i,j\in S$, and $R$ is defined only when $A_{H}$ is invertible.
  The main result  of this section is to give an explicit formula of $R$  in the following proposition. For this, denote the $(k',k'')$ entry of $W \equiv (I-P')^{-1}$ by $w_{k',k''}$, where all the diagonal entries of $W$ are positive because $W' = (I - P)^{-1} = \sum_{n=0}^{\infty} P^{n}$.

  \begin{proposition}
    \label{pro:R}
    Define $Q \equiv \{Q_{i,j}; i,j \in S\}$ as
    \begin{align}\label{eq:Q1}
      Q_{i,j} = \sum_{k \in \sr{C}_{i}} m_{k} w_{k,\ell(j)}, \quad i, j \in S,
    \end{align}
    Then, $A_H$ is invertible if and only if $Q$ is invertible, and the $S \times S$-indexed reflection matrix $R$ is given by $Q^{-1}$ when either $A_{H}$ or $Q$ is invertible.
  \end{proposition}

To prove \pro{R}, the following lemma plays a key role. It will be proved in \app{proof-i-b}.
\begin{lemma}
\label{lem:I-B}
Let $F_{k,k'} = 1(k' \in H(k))$ for $k,k' \in \sr{K}$, and define matrix $F \equiv \{F_{k,k'}; k, k' \in \sr{K}\}$, then $(I -B)^{-1} = F$.
\end{lemma}

As a corollary to \lem{I-B}, we will give an expression for $A^{-1}$ by using $A^{-1} = (I - B)^{-1} M (I - P^{'})^{-1}$. 
\begin{corollary}
  \label{cor:A-inv}
The matrix  $A$ is invertible and its inverse is given by
\begin{align}
\label{eq:A-inv}
 [A^{-1}]_{k',k''} & = \sum_{k \in \sr{K}} 1(k \in H(k')) m_{k} w_{k,k''}  \qquad k', k'' \in \sr{K}.
\end{align}
\end{corollary}

\begin{proof}[Proof of \pro{R}]
  Consider the decomposition of $A$ in \eq{A-b}. Since $B_{KL}=0$,
  $A_L = (I-P')_L M_L^{-1}$ is invertible. Therefore Schur's
  complement $S_H=A_H-A_{HL}A_L^{-1}A_{LH}$ is well defined. Assume $A_H$ is invertible. The Schur's complement
  $S_L=A_L-A_{LH} A_H^{-1}A_{HL}$ is also well defined. Hence, assuming $A_H$ is invertible,
  by Schur's theorem, both $S_L$ and $S_H$ are invertible and
  \begin{align}\label{eq:A-inv2}
    \begin{pmatrix}
      A_{L} & A_{LH}\\
      A_{HL} & A_{H}
    \end{pmatrix}^{-1} = 
    \begin{pmatrix}
      S_{L}^{-1} & - S_{{L}}^{-1}A_{LH}A_{H}^{-1}\\
      -S_{{H}}^{-1}A_{HL}A_{L}^{-1} & S_{{H}}^{-1}.
    \end{pmatrix}.
  \end{align}
  Comparing \eq{A-inv} with \eq{A-inv2}, we see that
  \begin{align*}
    [ S_L^{-1}]_{\ell(i), \ell(j)} =  \sum_{k \in \sr{K}} 1(k \in H(\ell(i))) m_{k} w_{k,\ell(j)}=\sum_{k\in \sr{C}_i} m_k w_{k,\ell(j)} = Q_{i,j}, \qquad i,j \in S.
  \end{align*}
  Hence, $Q$ is invertible, and $R = Q^{-1}$ because $R_{i,j} = [S_{L}]_{\ell(i),\ell(j)} = [Q^{-1}]_{i,j}$ for $i,j \in S$. Conversely, assume $Q$ is invertible. Then $[A^{-1}]_{\ell(i),\ell(j)} = Q_{i,j}$ by \eq{A-inv}, and hence $[A^{-1}]_{L}$ is invertible. 
Since $A$ is invertible, there exist sub-block matrices $X_{L}, X_{LH}, X_{H,L}, X_{H}$  such that
  \begin{align*}
    \begin{pmatrix}
      A_{L} & A_{LH}\\
      A_{HL} & A_{H}
    \end{pmatrix}  
    \begin{pmatrix}
     X_{L} & X_{LH}\\
      X_{HL} & X_{H}
    \end{pmatrix}
    = \begin{pmatrix}
    I_{H} & 0_{LH}\\
    0_{HL} & I_{L}
    \end{pmatrix},
  \end{align*}
  where the subscripts $L, LH, HL, H$ specify sub-blocks of $K \times K$ matrices. This implies
\begin{align*}
  A_{HL} X_{L} + A_{H} X_{HL} = 0_{HL}, \qquad A_{HL} X_{LH} + A_{H} X_{H} = I_{H}.
\end{align*}
Since $X_{L} = [A^{-1}]_{L}$, $X_{L}$ is invertible, so we have $A_{HL} = - A_{H} X_{HL} X_{L}^{-1}$ from the first equation. Substituting this $A_{HL}$ into the second equation, we have
\begin{align*}
  A_{H} ( - X_{HL} X_{L}^{-1} X_{LH} + X_{H}) = I_{H}.
\end{align*}
This proves that $A_{H}$ is invertible.
\end{proof}

\begin{corollary}
\label{cor:2-station}
For any 2-station $K$-class network with SBP discipline, the reflecting matrix $R$ is defined and an $\sr{M}$-matrix (therefore tight and completely-$\sr{S}$ by \thm{1} and also by \thm{2}) if $Q$ of \eq{Q1} has a positive determinant. In particular, if $R$ is defined and completely-$\sr{S}$, then $R$ is tight.
\end{corollary}
\begin{proof}
  Since it is assumed that the determinant $|Q| > 0$, $Q$ and so
  $A_{H}$ are invertible and $R = Q^{-1}$ by \pro{R}. Since $Q$ is a
  nonnegative matrix with positive diagonal entries by \eq{Q1} and has
  a positive determinant, $R=Q^{-1}$ is an $\sr{M}$-matrix for
  $d=2$. For the last claim, assume $R$ is defined and
  completely-$\sr{S}$, then we only need to show that $|Q|>0$. Since
  $R$ is defined, $A_{H}$ is invertible, so $Q$ is invertible and
  $R^{-1} = Q$ is a nonnegative matrix. From this and for $R$ to be
  completely-$\sr{S}$, we have $|Q|>0$ for $d=2$. Otherwise,
  $R=Q^{-1}$ has negative diagonal entries, which contradicts $R$ being
  completely-$\sr{S}$.
\end{proof}

\subsection{Reentrant lines with LBFS discipline}
\label{sec:reentrant_LBFS}

A reentrant line is a special case of the multiclass network in which all customers follow the same route, where a route is the ordered sequence of customer classes which are assigned at stations. The routing probability matrix $P$ of the reentrant line can be given by
\begin{align}
\label{eq:reentrant-P}
  P_{k,k'} = \left\{
\begin{array}{ll}
 1, \quad & k'=k+1, \; k=1,2,\ldots,K-1,  \\
 0, & \mbox{otherwise}.  
\end{array}
\right.
\end{align}
 From \eq{reentrant-P},
\begin{align}
\label{eq:I-P-1}
W = (I-P^{'})^{-1} = \sum_{n=0}^{K} (P')^{n} = \{1(k \le \ell); k,\ell \in \sr{K}\}.
\end{align}
For $i,j\in S$,
\begin{align}\label{eq:A2}
  \sum_{k \in \sr{C}_{i}} m_{k} w_{k,\ell(j)} = \sum_{k\in \sr{C}_i, k\ge \ell(j)} m_k,
\end{align}
which identical to the expression in (A.2) of
\cite{DaiYehZhou1997}.

\begin{theorem}
\label{thm:LBFS}
   For a reentrant line with LBFS discipline, $A_H$ is invertible and  the $R$ is a tight matrix.
\end{theorem}
\begin{proof}
 It follow from the proof of Theorem 3.1 (Part
b) in \cite{DaiYehZhou1997} that under LBFS the matrix $Q \equiv \{\sum_{k\in \sr{C}_i, k\ge \ell(j)} m_k; i,j \in S\}$ is a
$\sr{P}$-matrix, and therefore it is invertible, and $Q^{-1}$ is also a
$\sr{P}$-matrix by \cite{BermPlem1994}. Hence, by \pro{R},
$R=Q^{-1}$ is a $\sr{P}$-matrix. By the proof of Theorem 3.2 (Part b) in \cite{DaiYehZhou1997} that $R=\diag(a_{11}, \ldots, a_{dd})^{-1} R_1$, where $R_1$ is given as the matrix in the display below (A.12) of \cite{DaiYehZhou1997}. Note that the structure of the matrix $R_1$ 
fits the one in \thm{2}. Therefore, \thm{LBFS} follows from \thm{2}.
\end{proof}

\subsection{A reentrant line example with FBFS discipline}
\label{sec:reentrant_FBFS}

\begin{example}[Reentrant lines]
\label{exa:reentrant}

Consider the $7$-class and $3$-station reentrant line. Assume that $\sr{C}_{1} = \{1,2\}$, $\sr{C}_{2} = \{3,5\}$ and $\sr{C}_{3} = \{4,6,7\}$. We denote a customer arriving at station $i$ as class $k$ by $\os{k}{\to} [i]$. We assume that $\lambda_{1} = 1/3$, $m_1 = 2$, $m_2=1$, $m_3=2$, and $m_4=m_5=m_6=m_7=1$, and take the following route.
\begin{align}
\label{eq:ex-routing1}
 \mbox{In} \os{1}{\to} [1]  \os{2}{\to} [1] \os{3}{\to} [2] \os{4}{\to} [3] \os{5}{\to} [2] \os{6}{\to} [3] \os{7}{\to} [3] \to \mbox{Out}.
\end{align}
Hence, the following heavy traffic condition is satisfied for any assignment of priority because
\begin{align*}
  &\lambda_1 \sum_{k \in \sr{C}_{1}} m_{k} = \lambda_1(  m_{1} + m_{2}) = 1, \quad  \lambda_1\sum_{k \in \sr{C}_{2}} m_{k} =  \lambda_1(m_{3} + m_{5}) = 1, \quad \\
  & \lambda_1\sum_{k \in \sr{C}_{3}} m_{k} = \lambda_1( m_{4} + m_{6} + m_{7}) = 1.
\end{align*}

Under the FBFS discipline, the lowest priority classes in $\sr{C}_{1}$, $\sr{C}_{2}$ and $\sr{C}_{3}$ are $2$, $5$, $7$. Namely, $\ell(1) = 2$, $\ell(2) = 5$, and $\ell(3) = 7$. Hence, by \eq{A2},
\begin{align}
\label{eq:R-FBFS}
  R^{-1} =
\begin{pmatrix}
 1 & 0 & 0\\
 3 & 1 & 0\\
 3 & 2 & 1
\end{pmatrix}, \qquad
  R =
\begin{pmatrix}
 1 & 0 & 0\\
 -3 & 1 & 0\\
 3 & -2 & 1
\end{pmatrix}.
\end{align}
\end{example}

In \app{counterexample}, we will demonstrate that this $R$ is not tight. However, this does not mean that the stationary distribution and its associated probability distribution do not exist because the tight system condition is sufficient condition for their existence but may not be necessary.

\section{Proof of theorems}
\label{sec:proof-theorem}
\setnewcounter

In this section, we prove the theorems in Section \sect{main}.

\subsection{Proof of \thm{1}}
\label{sec:31}

Obviously, $R_{i,i} > 0$ for $i=1,2$ and $b > 0$ are necessary for $(R,b)$ to be a reflecting system. So, in this proof, we always assume this condition. In all the cases, the condition \eq{tight-c1} can be written as
\begin{align}
\label{eq:tight-ex1}
 & R \left(
\begin{array}{l}
 b_{1} (x^{(1)}_{\{1,2\}} - x_{\{1,2\}}) \\
 b_{2} (x^{(2)}_{\{1,2\}} - x_{\{1,2\}})   
\end{array}
\right) = 0,\\
\label{eq:tight-ex2}
 & R_{11} b_{1} (1 - x_{\{1\}}) + R_{12} b_{2} (x^{(2)}_{\{1\}} - x_{\{1\}}) = 0,\\
\label{eq:tight-ex3}
 & R_{21} b_{1} (x^{(1)}_{\{2\}} - x_{\{2\}}) + R_{22} b_{2} (1 - x_{\{2\}}) = 0,
\end{align}
where $x^{(1)}_{1} = x^{(2)}_{2} = 1$ is used to have \eq{tight-ex2} and \eq{tight-ex3}. Note that $R$ is a completely-$\sr{S}$ matrix for the case (\sect{main}.b) if and only if $R_{11} R_{22} - R_{12} R_{21} > 0$, which is automatically satisfied for the case (\sect{main}.c). Thus, for the cases (\sect{main}.b) and (\sect{main}.c), we always have $R_{11} R_{22} - R_{12} R_{21} > 0$, which implies that $R$ is invertible.

We now separately consider the three cases (\sect{main}.b)--(\sect{main}.d). We first consider the case (\sect{main}.b). Since $R$ is invertible, multiplying $R^{-1}$ to \eq{tight-ex1} from the left yields that $b_{j}(x^{(j)}_{\{1,2\}} - x_{\{1,2\}}) = 0$ for $j=1,2$, so we have
\begin{align}
\label{eq:tight-ex4}
  x^{(1)}_{\{2\}} = x^{(1)}_{\{1,2\}} = x_{\{1,2\}}, \qquad x^{(2)}_{\{1\}} = x^{(2)}_{\{1,2\}} = x_{\{1,2\}}.
\end{align}
If $R_{12}=R_{21}=0$, then $R$ is a diagonal matrix, so $(R,b)$ is a tight system for any $b>0$. Otherwise, $R_{12}<0$ or $R_{21}<0$. Assume that $R_{12}<0$, then \eq{tight-ex2} implies that $x_{\{1,2\}} - x_{\{1\}} \ge 0$ since $1 - x_{\{1\}} \ge 0$. But $x_{\{1,2\}} \le x_{\{1\}}$ by the assumption. Hence, we must have $x_{\{1,2\}} = x_{\{1\}}$, which together with \eq{tight-ex2} implies that $x_{\{1,2\}} = x_{\{1\}} = 1$. Furthermore, $x_{\{2\}} \ge x_{\{1,2\}} = 1$, so $x_{\{2\}} = 1$, which and \eq{tight-ex4} imply that $x^{(1)}_{\{2\}} = x^{(1)}_{\{1,2\}} = x_{\{1,2\}} = 1$ and $x^{(2)}_{\{1\}} = x^{(2)}_{\{1,2\}} = x_{\{1,2\}} = 1$. Clearly, the same results are obtained by assuming that $R_{21} < 0$. Hence, $(R,b)$ is a tight system for any $b > 0$ in the case (\sect{main}.b). We next consider the case (\sect{main}.c). In this case, we first consider the case that $R_{12} < 0$ and $R_{21} < 0$. Then, similarly to the previous case, \eq{tight-ex2} implies that $x^{(2)}_{\{1,2\}} = x^{(2)}_{\{1\}} \ge x_{\{1\}}$, while \eq{tight-ex3} implies that $x^{(1)}_{\{1,2\}} = x^{(1)}_{\{2\}} \ge x_{\{2\}}$. On the other hand, since $R$ is invertible, premultiplying $R^{-1}$ to \eq{tight-ex1} yields
\begin{align*}
 & x^{(1)}_{\{1,2\}} - x_{\{1,2\}} = 0, \qquad x^{(2)}_{\{1,2\}} - x_{\{1,2\}} = 0.
\end{align*}
Hence, $x^{(1)}_{\{1,2\}} = x_{\{1,2\}} $ and $x^{(2)}_{\{1,2\}} = x_{\{1,2\}}$, and therefore $x_{\{2\}} \le x^{(1)}_{\{1,2\}} = x_{\{1,2\}}$ and $x_{\{1\}} \le x^{(2)}_{\{1\}} = x_{\{1,2\}}$. Thus, $1 = x_{\{2\}} = x^{(1)}_{\{1,2\}} = x_{\{1,2\}}$ and $1 = x_{\{1\}} = x^{(2)}_{\{1\}} = x_{\{1,2\}}$ by the arguments in the previous case. Hence, $(R,b)$ is again a tight system for any $b > 0$.

We finally consider the case (\sect{main}.b). We show that $(R,b)$ can not be a tight system for any $b > 0$. To see this, we only need to find the solution $x_{D}, x^{(j)}_{D}$ such that they are in $(0,1]$. Let $\alpha_{1} = \frac {R_{12} b_{2}} {R_{11}b_{1}}$ and $\alpha_{2} = \frac {R_{21} b_{1}} {R_{22}b_{2}}$, then $\alpha_{1}, \alpha_{2} > 0$. Define, for $\varepsilon \in (0,1]$
\begin{align*}
  x_{\{1\}} = \frac {\varepsilon \alpha_{1} +1} {\alpha_{1} + 1}, \quad x_{\{2\}} = \frac {\varepsilon \alpha_{2} + 1} {\alpha_{2} + 1}, \qquad x_{\{1,2\}} = x^{(1)}_{\{2\}} = x^{(2)}_{\{1\}} = \varepsilon.
\end{align*}
Then, it can be checked that all the conditions in \dfn{tight} are satisfied, but $x_{D}, x^{(j)}_{D} \in (0,1]$ for $\emptyset \not= D \subset \{1,2\}$ and $j=1,2$. Hence, $(R,b)$ is not a tight system for any $b > 0$.

\subsection{Proof of \thm{2}}
\label{sec:32}

Because $R$ is a $\sr{P}$-matrix, $(R,b)$ is a reflecting system for $b>0$. Hence, we only need to prove that $(R,b)$ is tight. For this, we can assume $b=1$ without loss of generality because the condition \eq{R1} is unchanged when we replace $R$ by $R \diag (b)^{-1}$. To write the condition \eq{tight-c1} in a simpler form, recall that $\sr{N}_{d} = \{1,2,\ldots,d\}$, and let
\begin{align*}
  r_{D}^{(j)} = x_{D}^{(j)} - x_{D}, \qquad  j\in \sr{N}_{d}, \; D \subset \sr{N}_{d}.
\end{align*}
Then, for $D = \sr{N}_{d}$, \eq{tight-c1} can be written as
\begin{align*}
  R \begin{pmatrix} 
r_{\sr{N}_{d}}^{(1)} \\ 
\vdots \\ 
r_{\sr{N}_{d}}^{(d)}
\end{pmatrix} = 0.
\end{align*}
Since $R$ is a $\sr{P}$-matrix, $R$ is invertible, which implies that
\begin{align}
\label{eq:rJ1}
  r_{\sr{N}_{d}}^{(1)} = r_{\sr{N}_{d}}^{(2)} = r_{\sr{N}_{d}}^{(3)} = \ldots  = r_{\sr{N}_{d}}^{(\sr{N}_{d})} = 0, \; \mbox{ equivalently, } \; x_{\sr{N}_{d}}^{(1)} = x_{\sr{N}_{d}}^{(2)} = x_{\sr{N}_{d}}^{(3)} = \ldots  = x_{\sr{N}_{d}}^{(\sr{N}_{d})} = x_{\sr{N}_{d}}.
\end{align}
For $D = D_{1} \equiv \{2, 3, 4, \ldots , d\}$, it follows from \eq{tight-c1} that
\begin{align}
\label{eq:rD1}
& r_{D_{1}}^{(1)} = x_{D_{1}}^{(1)} - x_{D_{1}} = x_{\sr{N}_{d}}^{(1)} - x_{D_{1}} \leq x_{\sr{N}_{d}}^{(1)} - x_{\sr{N}_{d}} = r_{\sr{N}_{d}}^{(1)} = 0.
\end{align}
Let $e^{(1)}_{D_{1}}$ be the $d-1$-dimensional unit vector whose 1st entry is $1$, and let
\begin{align*}
  R_{D_{1},D_{1}} = \begin{pmatrix}
R_{2,2} & R_{2,3} & \cdots & R_{2,d}\\
R_{3,2} & R_{3,3} & \cdots & R_{3,d}\\
0 & R_{4,3} & \cdots & R_{4,d}\\
\vdots & \vdots & \ddots & \vdots\\
0 & 0 & \cdots & R_{d,d}
\end{pmatrix},
\qquad 
r^{(2,d)}_{D_{1}} = 
\begin{pmatrix}
r_{D_{1}}^{(2)}\\
r_{D_{1}}^{(3)}\\
r_{D_{1}}^{(4)}\\
\vdots\\
r_{D_{1}}^{(d)}
\end{pmatrix},
\end{align*}
then, from \eq{tight-c1}, we have
\begin{align}
\label{eq:rD12}
&  R_{2,1} r_{D_{1}}^{(1)} e^{(1)}_{D_{1}} + R_{D_{1},D_{1}} r^{(2,d)}_{D_{1}} = 0.
\end{align}
Multiplying the above equation by the row vector
$(r^{(2,d)}_{D_{1}})' \equiv \left(r_{D_{1}}^{(2)} r_{D_{1}}^{(3)} \cdots r_{D_{1}}^{(d)}\right)$
from the left, we have
\begin{align}
\label{eq:rD13}
  R_{2,1}r_{D_{1}}^{(1)} r_{D_{1}}^{(2)} + r^{(2,d)}_{D_{1}} R_{D_{1},D_{1}} (r^{(2,d)}_{D_{1}})' = 0.
\end{align}
Since $r_{D_{1}}^{(1)} \le 0$ by \eq{rD1}, if $r_{D_{1}}^{(1)} \not= 0$, then $r^{(2)}_{D_{1}} < 0$ by \eq{rD13} because $R_{2,1} r_{D_{1}}^{(1)} >  0$ by \eq{R1} and $r^{(2,d)}_{D_{1}} R_{D_{1},D_{1}} (r^{(2,d)}_{D_{1}})' > 0$ which follows from the fact that $R$ is a $\sr{P}$-matrix. Otherwise, $r_{D_{1}}^{(1)} = 0$ and \eq{rD12} imply that
\begin{align}
\label{eq:rD1-0}
  r_{D_{1}}^{(1)} = r_{D_{1}}^{(2)} = \ldots  = r_{D_{1}}^{(d)} = 0.
\end{align}
Let $D_{n} = \{n+1,n+2,\ldots,d\}$ for $n = 2,3,\ldots,d-1$ and $D_{d} = \emptyset$. Assume that $r_{D_{1}}^{(1)} < 0$, then $r_{D_{1}}^{(2)} < 0$, so $r_{D_{2}}^{(2)} = x^{(2)}_{D_{2}} - x_{D_{2}} = x^{(2)}_{D_{1}} - x_{D_{2}} \le x^{(2)}_{D_{1}} - x_{D_{1}} = r^{(2)}_{D_{1}} < 0$.

We apply the above arguments for $r^{(2)}_{D_{2}}$ instead of $r^{(1)}_{D_{1}}$ and the submatrix:
\begin{align*}
 R_{D_{2},D_{2}} \equiv \begin{pmatrix}
R_{3, 3} & R_{3, 4} & \cdots & R_{3, d}\\
R_{4,3} & R_{4, 4} & \cdots & R_{4, d}\\
\vdots & \vdots & \ddots & \vdots\\
0 & 0 &\cdots & R_{d, d}
\end{pmatrix},
\end{align*}
then $r_{D_{2}}^{(2)} < 0$ implies that $r_{D_{2}}^{(3)} < 0$. Repeating this arguments for $n=3, \ldots, d-1$, we have 
\begin{align*}
  r^{(n+1)}_{D_{n}} < 0, \qquad n=1,2,\ldots,d-1.
\end{align*}
This further implies that Since $r^{(d)}_{D_{d-1}} = x^{(d)}_{\emptyset} - x_{\{d\}} = 1 - x_{\{d\}} \ge 0$, this is a contradiction. Hence, similarly to \eq{rD1-0}, we must have
\begin{align}
\label{eq:rDd-0}
  r_{D_{i}}^{(1)} = r_{D_{i}}^{(2)} = \ldots  = r_{D_{i}}^{(d)} = 0, \qquad i=2,3,\ldots,d-1,
\end{align}
which implies that $x_{D_{d-1}} =  x^{(d)}_{D_{d-1}} = x^{(d)}_{\emptyset} = 1$.
This further implies that $x^{(d-1)}_{D_{d-2}} = x^{(d-1)}_{D_{d-1}} = x_{D_{d-1}} = 1$, which implies that $x_{D_{d-2}} = 1$. Thus, inductively, we have $x^{(i)}_{\sr{N}_{d}} = x_{\sr{N}_{d}} = 1$ for $i \in \sr{N}_{d}$ by \eq{rJ1}. This proves \thm{1} because $x^{(i)}_{D} \ge x^{(i)}_{\sr{N}_{d}}$ and $x_{D} \ge x_{\sr{N}_{d}}$ for any $D \subset \sr{N}_{d}$.

\appendix

\setnewcounter
\setcounter{section}{1}

\subsection{Proof of \lem{I-B}}
\label{app:proof-i-b}

\begin{proof}[Proof of \lem{I-B}]
  Define $K \times K$ matrix $G = \{G_{k,k'}; k,k' \in \sr{K}\}$ as $G = F(I - B)$, and let $\delta_{k,k'} = 1(k=k')$. Then, from the definition \eq{B} of $B$, we have
\begin{align*}
  G_{k,k'} = \sum_{j=1}^{K} F_{k,j} (I-B)_{j,k'} = \sum_{j=1}^{K} 1(j \in H(k)) (\delta_{j,k'} - 1(H(k')=H(j) \setminus \{j\})).
\end{align*}
Hence, for $k=k'$, $j \in H(k)$ implies that $1(H(k) = H(k') = H(j) \setminus \{j\})=0$, and therefore $G_{k,k'} = 1$ because $1(j \in H(k)) \delta_{j,k} = 1$ only if $j=k$. On the other hand, for $k \not= k'$, $\sum_{j=1}^{K} 1(j \in H(k)) \delta_{j,k'} = 1(k' \in H(k))$, and
\begin{align*}
 & \sum_{j=1}^{K} 1(j \in H(k)) 1(H(k')=H(j) \setminus \{j\}) \\
 & \quad = 1(k \in H(k)) 1(H(k')=H(k) \setminus \{k\}) + \sum_{j \in \sr{K} \setminus \{k\}} 1(j \in H(k)) 1(H(k')=H(j) \setminus \{j\})\\
 & \quad = 1(H(k')=H(k) \setminus \{k\}) + \sum_{j \in \sr{K} \setminus \{k\}} 1(j \in H(k)) 1(H(k')=H(j) \setminus \{j\})\\
 & \quad = 1(k' \in H(k)),
\end{align*}
which imply that $G_{k,k'} = 0$ for $k \not=k'$. Hence,
 $G$ is an identical matrix, and therefore $F$ is the inverse of $I-B$.

\end{proof}

\subsection{Counterexample for the FBFS case}
\label{app:counterexample}

We consider the reentrant line with FBFS discipline in \exa{reentrant}. For this model, $R$ is obtained by \eq{R-FBFS}.
Then, let $b_{j} = 1$ for $j=1,2,3$, and let
\begin{align*}
 & x_{\{1,2,3\}} = x^{(1)}_{\{1,2,3\}} = x^{(2)}_{\{1,2,3\}} = x^{(3)}_{\{1,2,3\}} = 1/2,\\
 & x_{\{1,2\}} = x^{(1)}_{\{1,2\}} = x^{(2)}_{\{1,2\}} = x^{(3)}_{\{1,2\}} = 1/2,\qquad
  x_{\{1,3\}} = x^{(1)}_{\{1,3\}} = x^{(2)}_{\{1,3\}} = x^{(3)}_{\{1,3\}} = 1/2,\\
 & x_{\{2,3\}} = x^{(1)}_{\{2,3\}} = x^{(2)}_{\{2,3\}} = x^{(3)}_{\{2,3\}} = 1/2,\\
 & x_{\{1\}} = x^{(1)}_{\{1\}} = x^{(2)}_{\{1\}} = 1, \quad x^{(3)}_{\{1\}} = 1/2, \qquad
  x_{\{2\}} = x^{(1)}_{\{2\}} = x^{(2)}_{\{2\}} = 1, \quad x^{(3)}_{\{2\}} = 1/2,\\
 & x_{\{3\}} = 3/4, \quad x^{(1)}_{\{3\}} = x^{(2)}_{\{3\}} = 1/2, \quad x^{(3)}_{\{3\}} = 1,\qquad
  x_{\emptyset} = x^{(1)}_{\emptyset} = x^{(2)}_{\emptyset} = x^{(3)}_{\emptyset} = 1,
\end{align*}
then it is seen that $\{x_{D}; D \subset \sr{N_{d}}\}$ and $\{x^{(i)}_{D}; i=1,2,3, D \subset \sr{N_{d}}\}$ satisfy all the conditions \eq{tight-c1}--\eq{tight-c3} for the tight system, but this set of solutions contradicts \eq{xD-s}. Hence, this $R$ is not tight. 


\def\cprime{$'$} \def\cprime{$'$} \def\cprime{$'$} \def\cprime{$'$}
  \def\cprime{$'$} \def\cprime{$'$} \def\cprime{$'$}

\end{document}